\documentclass{article}
\usepackage{amssymb,amsfonts,amsmath,amsthm}
\usepackage{hyperref}
\usepackage{parskip}
\usepackage{setspace}
\usepackage{fancyhdr}
\usepackage[dvipsnames]{xcolor}
\usepackage[numbers]{natbib}
\usepackage[bottom]{footmisc}
\usepackage[utf8]{inputenc}
\usepackage[english]{babel}
\usepackage{xcolor}

\newcommand\myshade{85}
\colorlet{mylinkcolor}{blue}
\colorlet{mycitecolor}{red}
\colorlet{myurlcolor}{Aquamarine}
\hypersetup{
	linkcolor  = mylinkcolor,
	citecolor  = mycitecolor,
	urlcolor   = myurlcolor!\myshade!black,
	colorlinks = true,
}
\title{Analytic expressions for some Mellin transforms with their application to prime counting function and interpolation formulas for the zeta function}

\author{Omprakash Atale\footnote{Khandesh College Education Society's
Moolji Jaitha College Jalgaon-425001, Maharashtra, India. E-mail: atale.om@outlook.com}} 
\date{February 26, 2022}

\begin{document}

\maketitle

\begin{abstract}
The aim of our present work here is to present few results in the theory of Mellin transforms using the method that S. Ramanujan used in proving his Master Theorem. Further applications of our results for some number-theoretic functions such as the prime counting function and the zeta function are established. 
\end{abstract}

\newtheorem{theorem}{Theorem}[section]
\newtheorem{corollary}{Corollary}[theorem]
\newtheorem{lemma}[theorem]{Lemma}
\theoremstyle{definition}
\newtheorem{definition}{Definition}[section]
\renewcommand\qedsymbol{$\blacksquare$}

\section{Introduction}
\textbf{1.1.} This is a technical paper which is an application based extension of Ramanujan's Master Theorem which is a powerful tool for evaluating Mellin type integrals [1129, pg. \cite{five}]. It states that if $f$ has expansion of the form 
\begin{equation*}
f\left( x \right) = \sum\limits_{n = 0}^\infty  {{{\left( { - 1} \right)}^n}\frac{{\phi \left( n \right)}}{{n!}}{x^n}}\tag{1.1}    
\end{equation*}
where $\phi(n)$ has a natural and continuous extension such that $\phi(0)\neq0$, then for $s>0$, we have
\begin{equation*}
\int\limits_0^\infty  {{x^{s - 1}}} \left( {\sum\limits_{n = 0}^\infty  {{{\left( { - 1} \right)}^n}\frac{{\phi \left( n \right)}}{{n!}}{x^n}} } \right)dx = \Gamma \left( s \right)\phi \left( { - s} \right).\label{oneptwo}\tag{1.2}
\end{equation*}
where $s$ is any positive integer. Eqn. (\ref{oneptwo}) was communicated by Ramanujan  in his \textit{Quarterly Reports} [\cite{one}, p.298]\cite{two} and was used by him in computing the values of certain definite integrals \cite{three}. We kindly request readers to make themselves familiar with the derivation of Ramanujan's Master Theorem from  [\cite{one}, p.298]\cite{two} whose method of proof is frequently used throughout the paper. Now, for the purpose of application, consider the following binomial expansion for $a, v>0$
\begin{equation*}
{\left( {1 + ax} \right)^{ - v}} = \sum\limits_{n = 0}^\infty  {{a^n}\frac{{\Gamma \left( {v + n} \right)}}{{\Gamma \left( v \right)}}} \frac{{{{\left( { - x} \right)}^n}}}{{n!}}.\tag{1.6}    
\end{equation*}
Employing Eqn. (\ref{oneptwo}) yields
\begin{equation*}
\int\limits_0^\infty  {{x^{n - 1}}} {\left( {1 + ax} \right)^{ - v}}dx = \frac{{\Gamma \left( n \right)\Gamma \left( {v - n} \right)}}{{{a^n}\Gamma \left( v \right)}}.\tag{1.7}    
\end{equation*}
Further applications and examples of Eqn. (\ref{oneptwo}) can be found in \cite{two}\cite{four}\cite{eight}.

\section{On certain Mellin transforms and their analytic expressions}
\textbf{2.1.} In this section, few theorems are established that are motivated by Ramanujan's method of deriving Eqn. (\ref{oneptwo}). Furthermore, certain applications of respective theorems are studied and applied in calculating the Mellin transform of certain infinite series. Throughout this section, it is assumed that $\phi(n)$ has a natural and continuous extension such that $\phi(0)\neq0$.
\begin{theorem}
If $0<\Re(s)<1$, then\newline
(i)
\begin{equation*}
\int\limits_0^\infty  {{x^{s - 1}}\sum\limits_{n = 0}^\infty  {\phi \left( {2n + 1} \right)\frac{{{{\left( { - 1} \right)}^{n}}}}{{\left( {2n + 1} \right)!}}{x^{2n + 1}}dx = \phi \left( { - s} \right)\Gamma \left( s \right)\sin \frac{{\pi s}}{2}} } ,\label{twopone}\tag{2.1}    
\end{equation*}
(ii)
\begin{equation*}
\int\limits_0^\infty  {{x^{s - 1}}\sum\limits_{n = 0}^\infty  {\phi \left( {2n} \right)\frac{{{{\left( { - 1} \right)}^{n}}}}{{\left( {2n} \right)!}}{x^{2n}}dx = \phi \left( { - s} \right)\Gamma \left( s \right)\cos \frac{{\pi s}}{2}} } ,\label{twoptwo}\tag{2.2}    
\end{equation*}
\end{theorem}
\begin{proof}
Consider the following Mellin transform of $\sin(ax)$ [\cite{seven}, pg. 332]
\begin{equation*}
\int\limits_0^\infty  {{x^{s - 1}}\sin (ax) dx}  = {a^{ - s}}\Gamma \left( s \right)\sin \left( {\frac{{\pi s}}{2}} \right).\tag{2.3}    
\end{equation*}
Substituting $a=r^{k}$ with $r>0$ in the above equation and expand $\sin(ax)$ in its Maclaurin series to get
\begin{equation*}
\int\limits_0^\infty  {{x^{s - 1}}\sum\limits_{n = 0}^\infty  {\frac{{{{\left( { - 1} \right)}^{n}}}}{{\left( {2n + 1} \right)!}}{{\left( {{r^k}x} \right)}^{2n + 1}}dx}  = \Gamma \left( s \right){r^{ - sk}}\sin \left( {\frac{{\pi s}}{2}} \right)}.     
\end{equation*}
Multiply both sides by ${\frac{{{f^{\left( k \right)}}\left( a \right){h^k}}}{{k!}}}$ where $f$ shall be specified later and sum on $k$, $0\leq{k}<\infty$ to get
\begin{equation*}
\sum\limits_{k = 0}^\infty  \frac{{{f^{\left( k \right)}}\left( a \right){h^k}}}{{k!}}\int\limits_0^\infty  {x^{s - 1}}\sum\limits_{n = 0}^\infty  {\frac{{{{\left( { - 1} \right)}^{n}}}}{{\left( {2n + 1} \right)!}}{{\left( {{r^k}x} \right)}^{2n + 1}}dx} \end{equation*}
\begin{equation*}
= \sum\limits_{k = 0}^\infty  {\frac{{{f^{\left( k \right)}}\left( a \right){{\left( {h{r^{ - s}}} \right)}^k}}}{{k!}}} \Gamma \left( s \right)\sin \left( {\frac{{\pi s}}{2}} \right),
\end{equation*}
\begin{equation*}
\int\limits_0^\infty  {x^{s - 1}}\sum\limits_{n = 0}^\infty  {\sum\limits_{k = 0}^\infty  {\frac{{{f^{\left( k \right)}}\left( a \right){{\left( {h{r^{2n + 1}}} \right)}^k}{{\left( { - 1} \right)}^{n}}}}{{k!\left( {2n + 1} \right)!}}{x^{2n + 1}}dx} }
\end{equation*}
\begin{equation*}
= \sum\limits_{k = 0}^\infty  {\frac{{{f^{\left( k \right)}}\left( a \right){{\left( {h{r^{ - s}}} \right)}^k}}}{{k!}}} \Gamma \left( s \right)\sin \left( {\frac{{\pi s}}{2}} \right).     
\end{equation*}
Now, let
\begin{equation*}
\phi \left( { - s} \right) = f\left( {h{r^{ - s}} + a} \right) = \sum\limits_{k = 0}^\infty  {\frac{{{f^{\left( k \right)}}\left( a \right){{\left( {h{r^{ - s}}} \right)}^k}}}{{k!}}}.     
\end{equation*}
Therefore, after further simplification we get
\begin{equation*}
\int\limits_0^\infty  {{x^{s - 1}}\sum\limits_{n = 0}^\infty  {\phi \left( {2n + 1} \right)\frac{{{{\left( { - 1} \right)}^{n}}}}{{\left( {2n + 1} \right)!}}{x^{2n + 1}}dx}  = \phi \left( { - s} \right)\Gamma \left( s \right)\sin \left( {\frac{{\pi s}}{2}} \right)}.   
\end{equation*}
Proof of (ii) can be obtained by a similar method using [\cite{six}, pg. 332] 
\begin{equation*}
\int\limits_0^\infty  {{x^{s - 1}}\cos (ax)dx = \Gamma \left( s \right){a^{ - s}}\cos \left( {\frac{{\pi s}}{2}} \right)}.\tag{2.4}     
\end{equation*}
\end{proof}
\begin{corollary}
For $0<\Re(s)<1$ and $s\neq{\frac{1}{2}}, \frac{1}{4}$ we have\newline
(i)
\begin{equation*}
\int\limits_0^\infty  {{x^{s - 1}}\left( {\zeta \left( 4 \right)x - \zeta \left( 8 \right)\frac{1}{{3!}}{x^3} + \zeta \left( {12} \right)\frac{1}{{5!}}{x^5} + ...} \right)} dx = \zeta \left( {2 - 2s} \right)\Gamma \left( s \right)\sin \left( {\frac{{\pi s}}{2}} \right),\label{twopfive}\tag{2.5}    
\end{equation*}
(ii)
\begin{equation*}
\int\limits_0^\infty  {{x^{s - 1}}\left( {\zeta \left( 2 \right) - \zeta \left( 6 \right)\frac{1}{{2!}}{x^2} + \zeta \left( {10} \right)\frac{1}{{4!}}{x^4} + ...} \right)} dx = \zeta \left( {2 - 2s} \right)\Gamma \left( s \right)\cos \left( {\frac{{\pi s}}{2}} \right),\label{twopsix}\tag{2.6}  
\end{equation*}
where $\zeta(s)$ is the zeta function.
\end{corollary}
\begin{proof}
Let
\begin{equation*}
\phi \left( n \right) = \zeta \left( {2n + 2} \right)    
\end{equation*}
which yields 
\begin{equation*}
\phi \left( {2n + 1} \right) = \zeta \left( {4n + 4} \right)    
\end{equation*}
and 
\begin{equation*}
\phi \left( {2n} \right) = \zeta \left( {4n + 2} \right).    
\end{equation*}
Now, use Eqn. (\ref{twopone}) and (\ref{twoptwo}) to get the desired result.
\end{proof}
\begin{corollary}
We have \newline
(i)
\begin{equation*}
\int\limits_0^\infty  {\zeta \left( 4 \right) - \zeta \left( 8 \right)\frac{1}{{3!}}{x^2} + \zeta \left( {12} \right)\frac{1}{{5!}}{x^4} + ...} dx = \frac{{{\pi ^3}}}{{12}},\label{twopeight}\tag{2.8} \end{equation*}
(ii)
\begin{equation*}
\int\limits_0^\infty  {\frac{{\log x}}{x}\left( {\zeta \left( 2 \right) - \zeta \left( 6 \right)\frac{1}{{2!}}{x^2} + \zeta \left( {10} \right)\frac{1}{{4!}}{x^4} + ...} \right)} dx = \frac{{{\pi ^4}}}{{24}}\label{twopnine},\tag{2.9} \end{equation*}
where $\zeta(s)$ is the zeta function.
\end{corollary}
\begin{proof}
We have mentioned earlier that $0<\Re(s)<1$, but their is a particular case where we can apply Theorem 2.1 at $s=0$. This can be done as follows. Using reflection formula for the gamma function of the right hand side of Eqn. (\ref{twopfive}) and taking the limit on $s$ the both sides to zero, we get
\begin{equation*}
\int\limits_0^\infty  {\zeta \left( 4 \right) - \zeta \left( 8 \right)\frac{1}{{3!}}{x^2} + \zeta \left( {12} \right)\frac{1}{{5!}}{x^4} + ...} dx = \frac{{{\pi ^2}}}{6}\mathop {\lim }\limits_{s \to 0} \frac{{\pi \sin \left( {\frac{{\pi s}}{2}} \right)}}{{\Gamma \left( {1 - s} \right)\sin \left( {\pi s} \right)}}.
\end{equation*}
Using L'Hospital's rule, we get
\begin{equation*}
\int\limits_0^\infty  {\zeta \left( 4 \right) - \zeta \left( 8 \right)\frac{1}{{3!}}{x^2} + \zeta \left( {12} \right)\frac{1}{{5!}}{x^4} + ...} dx = \frac{{{\pi ^3}}}{{12}}. \end{equation*}
Similarly, after calculating the value of integral (\ref{twopsix}), Eqn. (\ref{twopnine})  readily follows. A detailed proof can be found in Appendix A.1.
\end{proof}
\begin{corollary}
for $|t|<|a|$, $0<\Re(s)<1$ and $c>0$ We have
\begin{equation*}
\frac{1}{2}\int\limits_0^\infty  {{x^{s - 1}}\left[ {\zeta \left( {c,a + x} \right) - \zeta \left( {c,a - x} \right)} \right]dx }
\end{equation*}
\begin{equation*}
=\frac{(-1)^{-s}{\Gamma \left( s \right)\Gamma \left( {c - s} \right)}{\zeta \left( {c - s,a} \right)}}{{\Gamma \left( c \right)}}\sin \left( {\frac{{\pi s}}{2}} \right).\tag{2.10}   
\end{equation*}
\end{corollary}
\begin{proof}
From [\cite{nine}, pg. 412] we have 
\begin{equation*}
\sum\limits_{k = o}^\infty  {\frac{{{{\left( c \right)}_{2k + 1}}}}{{\left( {2k + 1} \right)!}}\zeta \left( {c + 2k + 1,a} \right){t^{2k + 1}} = \frac{1}{2}\left[ {\zeta \left( {c,a - t} \right) - \zeta \left( {c,a + t} \right)} \right]} \tag{2.11}   
\end{equation*}
where
\begin{equation*}
{\left( c \right)_{2k + 1}} = \frac{{\Gamma \left( {c + 2k + 1} \right)}}{{\Gamma \left( c \right)}}.
\end{equation*}
Letting $t=x$ and applying Eqn. (\ref{twopone}) yields the desire result.
\end{proof}
\begin{theorem}
We have
\begin{equation*}
\int\limits_0^\infty  {\frac{{\log x}}{x}\sum\limits_{n = 0}^\infty  {{c_n}{x^{2n}}d\pi \left( x \right) = \sum\limits_{n = 0}^\infty  {{c_n}{A_n}} } }\tag{2.12}     
\end{equation*}
where
\begin{equation*}
{c_n} = \zeta \left( {4n + 2} \right)\frac{{{{\left( { - 1} \right)}^n}}}{{\left( {2n} \right)!}}, \tag{2.13}   
\end{equation*}
\begin{equation*}
{A_n} = \sum\limits_{k = 1}^\infty  {\frac{{\mu \left( k \right)}}{{\left( {1 - 2n} \right)k - 1}} + f\left( {1 - 2n} \right)}, \tag{2.14}    
\end{equation*}
and $\pi(x)$ is the prime counting function.
\end{theorem}
\begin{proof}
Compress the sum in the integrand of Eqn. (\ref{twopeight}) and instead of integrating from all values from 0 to $\infty$, integrate only on primes [\cite{seven}, pg. 118, Eqn. (9.2)], that is  
\begin{equation*}
\int\limits_0^\infty  {\frac{{\log x}}{x}\sum\limits_{n = 0}^\infty  {\zeta \left( {4n + 2} \right)\frac{{{{\left( { - 1} \right)}^n}}}{{\left( {2n} \right)!}}{x^{2n}}} d\pi \left( x \right)}     
\end{equation*}
where $\pi(x)$ is the prime counting function. Therefore, we get
\begin{equation*}
\int\limits_0^\infty  {\frac{{\log x}}{x}\sum\limits_{n = 0}^\infty  {\zeta \left( {4n + 2} \right)\frac{{{{\left( { - 1} \right)}^n}}}{{\left( {2n} \right)!}}{x^{2n}}} d\pi \left( x \right)}  = \sum\limits_p {\frac{{\log p}}{p}} \sum\limits_{n = 0}^\infty  {\zeta \left( {4n + 2} \right)\frac{{{{\left( { - 1} \right)}^n}}}{{\left( {2n} \right)!}}{p^{2n}}}.     
\end{equation*}
Now, let
\begin{equation*}
{c_n} = \zeta \left( {4n + 2} \right)\frac{{{{\left( { - 1} \right)}^n}}}{{\left( {2n} \right)!}}    
\end{equation*}
then substitute the value of $c_{n}$ and inverting the order of summation, we get
\begin{equation*}
\int\limits_0^\infty  {\frac{{\log x}}{x}\sum\limits_{n = 0}^\infty  {{c_n}{x^{2n}}} d\pi \left( x \right)}  = \sum\limits_{n = 0}^\infty  {{c_n}\sum\limits_p {\frac{{\log p}}{p}} {p^{2n}} = \sum\limits_{n = 0}^\infty  {{c_n}\sum\limits_p {\frac{{\log p}}{{{p^{1 - 2n}}}}} } }.\label{twopfifteen}\tag{2.15}     
\end{equation*}
Now, using Eqn. (6.1) from [\cite{seven}, pg. 116], for $\Re(s)>1$, we have
\begin{equation*}
\sum\limits_p {\frac{{\log p}}{{{p^s}}} = \sum\limits_{k = 1}^\infty  {\frac{{\mu \left( k \right)}}{{sk - 1}} + f\left( s \right)} }     
\end{equation*}
where $\mu(k)$ is the Mobius function, $f(k)$ is analytic and is given by
\begin{equation*}
f\left( s \right) =  - \sum\limits_{k = 1}^\infty  {\mu \left( k \right)\left\{ {\frac{{\zeta '\left( {ks} \right)}}{{\zeta \left( {ks} \right)}} + \frac{1}{{ks - 1}}} \right\}}.     
\end{equation*}
Therefore, we get
\begin{equation*}
\sum\limits_p {\frac{{\log p}}{{{p^{1 - 2n}}}} = \sum\limits_{k = 1}^\infty  {\frac{{\mu \left( k \right)}}{{\left( {1 - 2n} \right)k - 1}} + f\left( {1 - 2n} \right)} }     
\end{equation*}
and
\begin{equation*}
f\left( {1 - 2n} \right) =  - \sum\limits_{k = 1}^\infty  {\mu \left( k \right)\left\{ {\frac{{\zeta '\left( {k\left( {1 - 2n} \right)} \right)}}{{\zeta \left( {k\left( {1 - 2n} \right)} \right)}} + \frac{1}{{k\left( {1 - 2n} \right) - 1}}} \right\}}.     
\end{equation*}
Substituting the above values in Eqn. (\ref{twopfifteen}) yields the following result
\begin{equation*}
\int\limits_0^\infty  {\frac{{\log x}}{x}\sum\limits_{n = 0}^\infty  {{c_n}{x^{2n}}d\pi \left( x \right) = \sum\limits_{n = 0}^\infty  {{c_n}{A_n}} } }     
\end{equation*}
where
\begin{equation*}
{A_n} = \sum\limits_{k = 1}^\infty  {\frac{{\mu \left( k \right)}}{{\left( {1 - 2n} \right)k - 1}} + f\left( {1 - 2n} \right)} .    
\end{equation*}
\end{proof}
\begin{theorem}
If $p, k, s>0$, then\newline
(i)
\begin{equation*}
\int\limits_0^\infty  {{x^{s - 1}}\sum\limits_{n = 0}^\infty  {\phi \left( {2n + 1} \right)\frac{{{{\left( { - 1} \right)}^{n}}}}{{\left( {2n + 1} \right)!{p^{2n + 1}}}}{x^{\left( {2n + 1} \right)k}}dx}  = \phi \left( {\frac{{ - s}}{k}} \right){{}_p}{\Gamma _k}\left( s \right)\sin \left( {\frac{{\pi s}}{{2k}}} \right)},\label{twopsixteen}\tag{2.16}     
\end{equation*}
(ii)
\begin{equation*}
\int\limits_0^\infty  {{x^{s - 1}}\sum\limits_{n = 0}^\infty  {\phi \left( {2n + 1} \right)\frac{{{{\left( { - 1} \right)}^{n}}}}{{\left( {2n + 1} \right)!{k^{2n + 1}}}}{x^{\left( {2n + 1} \right)k}}dx}  = \phi \left( {\frac{{ - s}}{k}} \right){\Gamma _k}\left( s \right)\sin \left( {\frac{{\pi s}}{{2k}}} \right)},\tag{2.17}
\end{equation*}
(iii)
\begin{equation*}
\int\limits_0^\infty  {{x^{s - 1}}\sum\limits_{n = 0}^\infty  {\phi \left( {2n} \right)\frac{{{{\left( { - 1} \right)}^{n}}}}{{\left( {2n} \right)!{p^{2n}}}}{x^{\left( {2n} \right)k}}dx}  = \phi \left( {\frac{{ - s}}{k}} \right){{}_p}{\Gamma _k}\left( s \right)\cos \left( {\frac{{\pi s}}{{2k}}} \right)},\tag{2.19}    
\end{equation*}
(iv)
\begin{equation*}
\int\limits_0^\infty  {{x^{s - 1}}\sum\limits_{n = 0}^\infty  {\phi \left( {2n} \right)\frac{{{{\left( { - 1} \right)}^{n}}}}{{\left( {2n} \right)!{k^{2n}}}}{x^{\left( {2n} \right)k}}dx}  = \phi \left( {\frac{{ - s}}{k}} \right){\Gamma _k}\left( s \right)\cos \left( {\frac{{\pi s}}{{2k}}} \right)},\tag{2.20}     
\end{equation*}
where ${_p{\Gamma _k}\left( s \right)}$ is the $p$-$k$ gamma function \cite{ten} and ${{\Gamma _k}\left( s \right)}$ is the $k$ gamma function \cite{eleven} \footnote{$ _{p}\Gamma_{k}(x) \Rightarrow\: _{k}\Gamma_{k}(x) = \Gamma_{k}(x)  $ as $ p = k $ and $ _{p}\Gamma_{k}(x) \Rightarrow \:_{1}\Gamma_{1}(x)= \Gamma (x)  $ as $  p,k\rightarrow 1 $.} defined as follows
\begin{equation*}
_p{\Gamma _k}(s) = {\left( {\frac{p}{k}} \right)^{\frac{s}{k}}}{\Gamma _k}(s) = \frac{{{p^{\left( {\frac{s}{k}} \right)}}}}{k}\Gamma \left( {\frac{s}{k}} \right).    
\end{equation*}
\end{theorem}
\begin{proof}
Replace $x$ with $x^{k}$/$p$ in Eqn. (\ref{twopsixteen}) to get
\begin{equation*}
\int\limits_0^\infty  {\frac{{{x^{sk - k}}}}{{{p^{s - 1}}}}\sum\limits_{n = 0}^\infty  {\phi \left( {2n + 1} \right)\frac{{{{\left( { - 1} \right)}^{n}}}}{{\left( {2n + 1} \right)!{p^{2n + 1}}}}{x^{\left( {2n + 1} \right)k}}\frac{{k{x^{k - 1}}}}{p}dx}  = \phi \left( { - s} \right)\Gamma \left( s \right)\sin \left( {\frac{{\pi s}}{2}} \right)},     
\end{equation*}
\begin{equation*}
\int\limits_0^\infty  {\frac{{{x^{sk - 1}}}}{{{p^s}}}\sum\limits_{n = 0}^\infty  {\phi \left( {2n + 1} \right)\frac{{{{\left( { - 1} \right)}^{n}}}}{{\left( {2n + 1} \right)!{p^{2n + 1}}}}{x^{\left( {2n + 1} \right)k}}kdx}  = \phi \left( { - s} \right)\Gamma \left( s \right)\sin \left( {\frac{{\pi s}}{2}} \right)}.     
\end{equation*}
Now, replacing $s$ with $s$/$k$ yields
\begin{equation*}
\int\limits_0^\infty  {{x^{s - 1}}\sum\limits_{n = 0}^\infty  {\phi \left( {2n + 1} \right)\frac{{{{\left( { - 1} \right)}^{n}}}}{{\left( {2n + 1} \right)!{p^{2n + 1}}}}{x^{\left( {2n + 1} \right)k}}dx}  = \phi \left( {\frac{{ - s}}{k}} \right)\frac{{{p^{\frac{s}{k}}}}}{k}\Gamma \left( {\frac{s}{k}} \right)\sin \left( {\frac{{\pi s}}{{2k}}} \right)}.     
\end{equation*}
By further simplification, the desired result readily follows. (iii) can be derived in a similar manner. (ii) and (iv) are special cases of (i) and (iii) when $p=k$ respectively.
\end{proof}

\section{Appendix}
\textbf{A.1.} Take the derivatives both the sides of Eqn. (\ref{twopsix}) with respect to $s$ and then multiply and divide right hand side of the equation with $\Gamma(s)$ to get
\begin{equation*}
\int\limits_0^\infty  {{x^{s - 1}}\log x\left( {\zeta \left( 2 \right) - \zeta \left( 6 \right)\frac{1}{{2!}}{x^2} + \zeta \left( {10} \right)\frac{1}{{4!}}{x^4}...} \right)dx}= \zeta '\left( {2 - 2s} \right)\Gamma \left( s \right)\cos \left( {\frac{{\pi s}}{2}} \right)     
\end{equation*}
\begin{equation*}
\frac{1}{{\Gamma \left( s \right)}}\zeta \left( {2 - 4s} \right)\psi \left( s \right)\cos \left( {\frac{{\pi s}}{2}} \right) + \frac{\pi }{2}\zeta \left( {2 - 2s} \right)\Gamma \left( s \right)\sin \left( {\frac{{\pi s}}{2}} \right).    
\end{equation*}
Using the reflection formula for gamma function, we get
\begin{equation*}
\int\limits_0^\infty  {{x^{s - 1}}\log x\left( {\zeta \left( 2 \right) - \zeta \left( 6 \right)\frac{1}{{2!}}{x^2} + \zeta \left( {10} \right)\frac{1}{{4!}}{x^4}...} \right)dx} = \zeta '\left( {2 - 2s} \right)\frac{{\pi \cos \left( {\frac{{\pi s}}{2}} \right)}}{{\Gamma \left( {1 - s} \right)\sin \pi s}}    
\end{equation*}
\begin{equation*}
+ \frac{1}{\pi }\zeta \left( {2 - 2s} \right)\Gamma \left( {1 - s} \right)\sin{(\pi s)}\psi \left( s \right)\cos \left( {\frac{{\pi s}}{2}} \right) + \frac{\pi }{2}\zeta \left( {2 - 2s} \right)\frac{{\pi \sin \left( {\frac{{\pi s}}{2}} \right)}}{{\Gamma \left( {1 - s} \right)\sin \pi s}}.    
\end{equation*}
Now, taking limit both the sides of $s$ from $s \to 0 $, and applying L' Hospital's rule, we get
\begin{equation*}
\int\limits_0^\infty  {\frac{{\log x}}{x}\left( {\zeta \left( 2 \right) - \zeta \left( 6 \right)\frac{1}{{2!}}{x^2} + \zeta \left( {10} \right)\frac{1}{{4!}}{x^4}...} \right)dx}     
\end{equation*}
\begin{equation*}
= \mathop {\lim }\limits_{s \to 0} \zeta '\left( {2 - 2s} \right)\frac{{\pi \cos \left( {\frac{{\pi s}}{2}} \right)}}{{\Gamma \left( {1 - s} \right)\sin \pi s}} + \mathop {\lim }\limits_{s \to 0} \frac{\pi }{2}\zeta \left( {2 - 2s} \right)\frac{{\pi \sin \left( {\frac{{\pi s}}{2}} \right)}}{{\Gamma \left( {1 - s} \right)\sin \pi s}},    
\end{equation*}
\begin{equation*}
= \zeta '\left( 2 \right)\mathop {\lim }\limits_{s \to 0} \frac{{\pi \cos \left( {\frac{{\pi s}}{2}} \right)}}{{\Gamma \left( {1 - s} \right)\sin \pi s}} + \frac{{{\pi ^2}}}{4}\zeta \left( 2 \right)\mathop {\lim }\limits_{s \to 0} \frac{{\cos \left( {\frac{{\pi s}}{2}} \right)}}{{\cos \pi s}}    
\end{equation*}
\begin{equation*}
= 0 + \frac{{{\pi ^2}}}{2}\zeta \left( 2 \right) = \frac{{{\pi ^4}}}{{24}} .   
\end{equation*}
Therefore,
\begin{equation*}
\int\limits_0^\infty  {\frac{{\log x}}{x}\left( {\zeta \left( 2 \right) - \zeta \left( 6 \right)\frac{1}{{2!}}{x^2} + \zeta \left( {10} \right)\frac{1}{{4!}}{x^4}...} \right)dx}  = \frac{{{\pi ^4}}}{{24}}.    
\end{equation*}

\end{document}